\documentclass[12pt,reqno]{amsart}

\usepackage{amsmath,amsthm,amssymb,color}
\usepackage{tikz,pgfplots}

\usepackage{a4}

\def\d{\,{\rm d}}
\def\div{{\rm div}}

\def\RR{\mathbb{R}}

\def\tX{\widetilde X}

\newtheorem{theorem}{Theorem}

\newtheorem{lemma}[theorem]{Lemma}

\def\XXint#1#2#3{{\setbox0=\hbox{$#1{#2#3}{\int}$ }
\vcenter{\hbox{$#2#3$ }}\kern-.6\wd0}}

\def\gmin{{g_{min}}}
\def\gmax{{g_{max}}}
\def\Bd{{B^{\delta}}}
\def\Td{{T^{\delta}}}

\begin{document}
\title[Numerical identification of a nonlinear diffusion law]{Numerical identification \\of a nonlinear diffusion law\\via regularization in Hilbert scales}
\author[H. Egger]{Herbert Egger$^\dag$}
\author[J.-F. Pietschmann]{Jan-Frederik Pietschmann$^\dag$}
\author[M. Schlottbom]{Matthias Schlottbom$^\dag$}
\thanks{$^\dag$Numerical Analysis and Scientific Computing, Department of Mathematics, TU Darmstadt, Dolivostr. 15, 64293 Darmstadt. \\
Email: {\tt $\{$egger,pietschmann,schlottbom$\}$@mathematik.tu-darmstadt.de}}

\begin{abstract}
  We consider the reconstruction of a diffusion coefficient in a quasilinear elliptic problem from a single measurement of overspecified Neumann and Dirichlet data. The uniqueness for this parameter identification problem has been established by Cannon and we therefore focus on the stable solution in the presence of data noise. For this, we utilize a reformulation of the inverse problem as a linear ill-posed operator equation with perturbed data and operators. 
  We are able to explicitly characterize the mapping properties of the corresponding operators which allow us to apply regularization in Hilbert scales. 
  We can then prove convergence and convergence rates of the regularized reconstructions under very mild assumptions on the exact parameter. These are, in fact, already needed for the analysis of the forward problem and no additional source conditions are required.
  Numerical tests are presented to illustrate the theoretical statements.
\end{abstract}

\maketitle


{\footnotesize
{\noindent \bf Keywords:} 
nonlinear diffusion,
parameter identification,
quasilinear elliptic problems, 
regularization in Hilbert scales
}


{\footnotesize
\noindent {\bf AMS Subject Classification:}  
35R30,65J22
}

\section{Introduction}
We consider the identification of an unknown coefficient $a=a(u)$ in the quasilinear elliptic problem
\begin{align}
  -\div(a(u)\nabla u) & =  0 \qquad \text{in } \Omega, \label{eq:quasi}\\
 a(u) \partial_n u &= j \qquad  \text{on } \partial \Omega, \label{eq:quasi_bc} 
\end{align}
from additional observation of boundary data
\begin{align} 
                 u &= g \qquad \text{on } \gamma \subset \partial \Omega. \label{eq:quasi_meas}
\end{align}
along a curve $\gamma \subset \partial \Omega$. 
Under mild assumptions on the parameter $a$, the data $j$, and the domain $\Omega$,
the boundary value problem \eqref{eq:quasi}--\eqref{eq:quasi_bc} has a continuously differentiable solution which is unique up to constants and sufficiently smooth such that the measurements \eqref{eq:quasi_meas} make sense.
Inverse problems of this kind arise for instance in steady state heat transfer, where $u$ denotes the temperature and $a(u)$ the thermal conductivity; see \cite{Alifanov,BeckBlackwellStClair,Cannon} for related inverse problems in heat transfer. Other possible applications are population dynamics or flow in porous media; see \cite{Gurney1975,Vazquez}.

The identifiability of $a(u)$ from \eqref{eq:quasi}--\eqref{eq:quasi_meas} on the range of states that are attained on $\gamma$ has been established by Cannon \cite{Cannon67}, see also \cite{Isakov93}. 
Other variants of single measurements of overspecified boundary data have been considered in \cite{Kuegler03,PilantRundell88}.
Using observation of the full Dirichlet-to-Neuman map, the identifiability of $a(u)$ in the presence of an additional unkown coefficient $c(x)$ in front of the zero order term has been established recently by the authors \cite{EggPieSch13}. For such rich measurements, even uniqueness for a parameter $a(x,u)$ depending on position and state can be proven \cite{Sun96}.
The identification of other state dependent parameters in semi-linear elliptic equations has been investigated for instance in \cite{Isakov93,IsakovSylvester94}. 
Related uniqueness and stability results for parabolic problems can be found in \cite{CannonDuChateau80,CannonYin89,DuChateauRundell85,DuChateau04,EggerEnglKlibanov05,Isakov93,Isakov93b,KaltenbacherKlibanov08,Klibanov04,LorenziLunardi90,PilantRundell86}; see also \cite{Isakov06,KlibanovTimonov04,Yamamoto09} for an overview and further results.

The inverse problem of determining the coefficient $a(u)$ from \eqref{eq:quasi}--\eqref{eq:quasi_meas} can be shown to be ill-posed when realistic topologies for the data error and the parameter space are used. Therefore, some regularization method is required for a stable solution.  
A direct application of Tikhonov regularization yields regularized approximations of the form
\begin{align} \label{eq:tik}
a_\alpha^\delta = \arg\min \|u(a) - g^\delta\|_{Y}^2 + \alpha \|a\|_{X}^2, 
\end{align}
where $g^\delta$ are the perturbed Dirichlet data and $u(a)$ is the solution of \eqref{eq:quasi}--\eqref{eq:quasi_bc} normalized such that $\int_\gamma u \d s = \int_\gamma g^\delta \d s$. 
A similar approach has been considered in \cite{InghamYuan93}.
The topology of the parameter space $X$ and the measurment space $Y$ have to be chosen suitably in order to make this regularization strategy well-defined and converging. 
Note that a numerical realization of \eqref{eq:tik} requires the minimization of a non-convex functional governed by a nonlinear partial differential equation. In addition, convergence of the regularized solutions with rates can only be guaranteed with an appropriate source condition. 
As shown by K\"ugler \cite{Kuegler03}, for a problem with slightly different boundary data, the standard source condition \cite{EHN96,EnKuNeu89} holds for the choice $X=H^1(I)$ and $Y=L^2(\gamma)$, provided that the exact solution $a$ is sufficiently smooth on the range $I$ of relevant states. Convergence rates $\|a-a_\alpha^\delta\|_{H^1(I)} = O(\sqrt{\delta})$ have been established in \cite{Kuegler03} for $a \in H^4(I)$ subject to some non-trivial additional boundary conditions. 
Let us remark that, based on the results of \cite{Kuegler03}, corresponding results can also be established for other, e.g.,  iterative regularization methods \cite{BakGon89,KalNeuSch08}.
Various numerical methods for related problems in inverse heat transfer have been investigated in \cite{AliArt75,Art80,EggHenMarMha09,LesEllIng96,LesIng95,Yang03}; see also \cite{Alifanov,BeckBlackwellStClair,Cannon} for an introduction and further references. 

In this paper, we propose a problem adapted strategy for the solution of the parameter identification problem for \eqref{eq:quasi}--\eqref{eq:quasi_meas}.
We utilize a reformulation of the inverse problem as a linear operator equation
\begin{align} \label{eq:linip}
T A = y, 
\end{align}
with operator $T$ and data $y$ depending on the boundary data $j$ and $g$, respectively. 
For ease of presentation, we utilize the antiderivative $A$ of the parameter $a$ here. 
A similar argument is used for the proof of uniqueness; see \cite{Cannon67} and Section~\ref{sec:main} below. 
We will explicitly characterize the smoothing properties of the operator $T$ which allow us to utilize Tikhonov regularization in Hilbert scales \cite{Nat84}; see also \cite{Egg06,Neu88,Tau96}.
Since we allow for perturbations in both boundary data \eqref{eq:quasi_bc} and \eqref{eq:quasi_meas},
we have to consider perturbations in the right hand side $y$ and in the operator $T$ simultaneously. 
Based on the characterization of the mapping properties of $T$, we are able to prove convergence rates 
under very weak assumptions on the parameter $a$, which are already required to make the forward problem \eqref{eq:quasi}--\eqref{eq:quasi_meas} well-defined.
\\

The outline of the paper is as follows:
In the next section, we introduce our basic assumptions and briefly state the main results. The proofs are then given in the following  sections. In Section~\ref{sec:solvability}, we establish solvability and regularity results for the governing boundary value problem \eqref{eq:quasi}--\eqref{eq:quasi_bc}. 
The forward operator $T$ of the linear inverse problem~\eqref{eq:linip} and its perturbation resulting from measurement errors 
are analyzed in Section~\ref{sec:map}.
We then introduce the Hilbert scale which is relevant for the problem under investigation in Section~\ref{sec:hs}, and we translate the mapping properties of $T$ into this scale.
In Section~\ref{sec:reghs}, we present the analysis of Tikhonov regularization in Hilbert scales incorporating perturbations in the operator and right hand side simultaneously. 
For illustrattion of our theoretical results, some numerical experiments are presented in Section~\ref{sec:num}.

\section{Statement of the main results} \label{sec:main}

Let us briefly sketch the derivation of the uniqueness result which serves as a motivation for our approach. 
As outlined by Cannon \cite{Cannon67}, the function
$a(u)$ can be determined uniquely on, and only on, the interval
\begin{align} \label{eq:I}
I = \{g(x) : x \in \gamma\} = [\gmin,\gmax] 
\end{align}
of states $u$ that are attained on $\gamma$. 
This can be shown with the following argument: 
Let $v(x)=A(u(x))$ where $A(u) = \int_{\gmin}^{u} a(w) \d w$ is the antiderivative of $a(u)$. 
Using the chain rule,  \eqref{eq:quasi}--\eqref{eq:quasi_meas} can be seen to be equivalent to the boundary value problem 
\begin{align}
-\Delta  v &= 0 \qquad \text{in } \Omega, \label{eq:bvp1}\\
\partial_n   v &= j \qquad \text{on } \partial \Omega, \label{eq:bvp2}
\end{align}
with overspecified Dirichlet data
\begin{align} \label{eq:bvp3}
 v(x) &= A(g(x)) \qquad \text{on } \gamma.
\end{align}
Note that $v$ is determined up to a constant via  \eqref{eq:bvp1}-\eqref{eq:bvp2} without knowledge of the parameter $a$.
Once $v$ has been computed, the unknown parameter can be determined by \eqref{eq:bvp3} alone. 
Formally differentiating this equation along $\gamma$ then yields 
the explicit reconstruction formula 
\begin{align} \label{eq:formula}
 \frac{d}{ds} v(\gamma(s))  = a(g(\gamma(s)) \frac{d}{ds} g(\gamma(s)), \qquad s \in [0,1],
\end{align}
which was the main result of \cite{Cannon67} and which implies uniqueness.\\

To make this derivation rigorous, let us pose some assumptions on the domain and the exact data:
\begin{itemize}
 \item[(A1)] $\Omega\subset\RR^3$, is a bounded domain with $C^{1,1}$ boundary.
 \item[(A2)] $j \in W^{1-1/p,p}(\partial \Omega)$ for some $p>3$ and $\int_{\partial \Omega} j \d s = 0$.
\end{itemize}
We only consider the three dimensional case here, but all results carry over easily to 
two dimensions. 
By $W^{s,p}(\Omega)$ we denote the usual Sobolev spaces of functions with broken derivatives in $L^p(\Omega)$, see \cite{Adams75},
and as usual we write $H^k(\Omega)=W^{k,2}(\Omega)$ for the corresponding Hilbert spaces.
In addition to the assumptions on the experimental setup, we require that the exact parameter satisfies
\begin{itemize}
 \item[(A3)] $a\in H^1(\RR)$ and $0 < c_a \leq a(u) \leq c_a^{-1}$ for all $u \in \RR$
\end{itemize}
for some constant $c_a$. 
Under such conditions, which are reasonable from a physical point of view, one can show 
\begin{theorem} \label{thm:solvability}
Let (A1)--(A3) hold. 
Then the boundary value problem \eqref{eq:quasi}--\eqref{eq:quasi_bc} has a solution $u \in W^{1,p}(\Omega) \subset  C^1(\overline{\Omega})$ which is unique up to constants.
\end{theorem}
The proof of this result is given in Section~\ref{sec:solvability}. 
As a consequence of the regularity of the solution, 
the additional measurement \eqref{eq:quasi_meas} is in fact well-defined,
and the following conditions, which we require for the analysis of the inverse problem, make sense:
\begin{itemize}
 \item[(A4)] $\gamma:[0,1] \to \partial \Omega$ is a $C^1$ curve, i.e., $0 < c_\gamma \le |\gamma'(s)| \le c_\gamma^{-1}$.
 \item[(A5)] $0 < c_g \le |g'(\gamma(s))| \le c_g^{-1}$ for all $s \in [0,1]$.
\end{itemize}
These conditions can be addressed by a reasonable measurement setup. 
It is actually sufficient to require them in a piecewise sense.\\

Now let $v$ be the solution of the transformed problem \eqref{eq:bvp1}--\eqref{eq:bvp2} with exact Neumann data,
and let us define an operator $T$ and data $y$ by
\begin{align} \label{eq:opT}
T: A \mapsto A(g(\gamma(\cdot))
\quad \text{and} \quad y(s) = v(\gamma(s)), \quad s \in [0,1].
\end{align}
Then condition \eqref{eq:bvp3} can be written equivalently as linear operator equation 
\begin{align} \label{eq:ip}
 T A = y \qquad \text{on} \quad  [0,1].
\end{align}
The mapping properties of the linear operator $T$ are characterized by
\begin{theorem} \label{thm:map}
 Let (A1)-(A5) hold. Then 
 \begin{align} \label{eq:mapT}
 c_T \|A\|_{L^2(I)} \le \|T A\|_{L^2(0,1)} \le C_T \|A\|_{L^2(I)}
 \end{align}
for all $A \in L^2(\Omega)$ with positive constants $c_T$, $C_T$. 
\end{theorem}
This statement is proven in Section~\ref{sec:map} and can be used 
to characterize the degree of ill-posedness of the inverse problem \eqref{eq:ip} depending on the choice of topologies.
We will later consider $T$ as an operator from $H^2$ to $L^2$, which yields an inverse 
problem that is, loosely speaking, as ill-posed as two times differentiation.\\

Let us now turn to the practically relevant case that the data are perturbed by the measurement process. 
Denote by $j^\delta$ and $g^\delta$ the measured boundary data
and let $\Td$ and $y^\delta$ be the corresponding operator and right hand side, respectively. 
The problem to be actually solved then reads
\begin{align} \label{eq:pip} 
\Td A = y^\delta \quad \text{on } [0,1].
\end{align}
For the analysis of this perturbed inverse problem, we will assume that the following conditions on the measurement errors hold:
\begin{itemize}
 \item[(A6)] $\|j - j^\delta\|_{L^2(\partial \Omega)} \le \delta$.
 \item[(A7)] $\|g - g^\delta\|_{L^2(\gamma)} \le \delta$ and $g^\delta(\gamma) \subset [\gmin,\gmax]$. 
\end{itemize}
The second condition on $g^\delta$ is made only to simplify our presentation. 
It can be guaranteed, e.g., by a bound for the noise in $L^\infty$ and appropriate truncation of the data. 
Using (A6)--(A7), we can estimate the perturbations in the operator $T$ and the right hand side $y$ in \eqref{eq:pip} as follows:
\begin{theorem} \label{thm:perT}
Let (A1)--(A7) hold. 
Then 
\begin{align} \label{eq:perturbation}
 \|(T - \Td) w\|_{L^2(0,1)} \le C \delta \|w\|_{H^2(I)} 
 \quad \text{and} \quad 
 \|y-y^\delta\|_{L^2(0,1)} \le C \delta
\end{align}
with some constant $C$ independent of the noise level $\delta$.
\end{theorem}
This result will be proven in Section~\ref{sec:map}.
Let us mention that, as a consequence of the theorem, the operator $\Td$ is a continuous mapping from $H^2(I)$ to $L^2(0,1)$, which will determine our choice of topologies. \\

For the stable solution of the inverse problem \eqref{eq:pip} with perturbed operator and data,
we then consider Tikhonov regularization 
\begin{align} \label{eq:tikhs}
A_\alpha^\delta = \arg\min_{…\stackrel{A \in H^2(I)}{A(\gmin)=0}} \|\Td A - y^\delta\|^2_{L^2(\gamma)} + \alpha \|A''\|_{L^2(I)}^2 + \alpha \|A\|_{L^2(I)}^2.
\end{align}
The choice of the set of admissible parameters and of the regularization term is driven by Theorem~\ref{thm:map} and \ref{thm:perT}.
Using the definition $A(u) = \int_{g_{min}}^u a(w) \d w$, one can formulate an equivalent problem 
for the parameter $a$, which is the form we use for our numerical tests. 
Since the Tikhonov functional in \eqref{eq:tikhs} is continuous, quadratic, and strictly convex,
the regularized solution $A_\alpha^\delta$ is uniquely defined. 
Moreover, the discretization and numerical minimization are straight forward \cite{KinNeu88}. 
Using the results of Theorem~\ref{thm:map} and \ref{thm:perT}, and arguments of regularization in Hilbert scales \cite{Egg06,Nat84,Neu88,Tau96}, we will prove in Section~\ref{sec:reghs} the following convergence rates result for the regularization method \eqref{eq:tikhs}.
\begin{theorem} \label{thm:rates}
Let (A1)-(A7) hold.
Then for $\alpha \approx \delta^2$, we have
\begin{align} \label{eq:rates}
 \|A-A_\alpha^\delta\|_{H^2(I)} = O(1) 
 \quad \text{and} \quad 
 \|A-A_\alpha^\delta\|_{L^2(I)} = O(\delta).
\end{align}
\end{theorem}
By interpolation, one can then get rates for the convergence in the $H^s$-norm with $0 < s \le 2$.
Note that the second estimate in \eqref{eq:rates} is actually a conditional well-posedness result.  
Differentiation of $A(u) = \int_{\gmin}^u a(w) \d w$ and interpolation finally yields our main result:
\begin{theorem} \label{thm:main}
Let the assumptions of Theorem~\ref{thm:rates} hold. Then
\begin{align} \label{eq:ratesa}
  \|a-a_\alpha^\delta\|_{H^1(I)} = O(1) 
 \quad \text{and} \quad 
 \|a-a_\alpha^\delta\|_{L^2(I)} = O(\sqrt{\delta}).
\end{align}
\end{theorem}
Let us shortly discuss the assertion of this theorem: Apart from assumption (A3), which was already needed to make the forward problem and the measurement setup well-defined, we do not require any additional (source) condition here. 
As a consequence of converse results \cite{EHN96}, we can therefore only get a uniform bound for the error in the $H^1$ norm, or arbitrarily slow convergence. Due to the precise characterization of the mapping properties in \eqref{eq:mapT}, we do however obtain quantitative convergence estimates in weaker norms. 
Using stronger source conditions, it is possible to obtain the first bound in \eqref{eq:rates} in a stronger norm and, as a consequence, to improve the convergence rate in the second estimate of Theorem~\ref{thm:main}. This follows by standard arguments of regularization in Hilbert scales, and we will comment on this in some detail in Section~\ref{sec:num}. \\

The remainder of the manuscript is devoted to the proofs of the results presented in this section.

\section{Solvability of the governing boundary value problem}
\label{sec:solvability}

Let us start with an auxiliary result concerning the well-posedness of the transformed boundary value problem \eqref{eq:bvp1}--\eqref{eq:bvp2}.
\begin{lemma} \label{lem:solv2}
Let (A1)--(A2) hold. 
Then \eqref{eq:bvp1}--\eqref{eq:bvp2} has a strong solution $v \in W^{2,p}(\Omega)$ which is 
continuously differentiable up to the boundary and unique up to constants.
In particular, there exists a unique solution with zero average.
\end{lemma}
\begin{proof}
The existence of a unique weak solution 
$v_0 \in H^1(\Omega)$ with zero average follows from the Lax-Milgram theorem, 
and the set of  solutions is given by $v_0 + \RR$. Since $j \in W^{1-1/p,p}(\partial \Omega)$ and since $\Omega$ is sufficiently regular, we obtain $v \in W^{2,p}(\Omega)$ with $p>3$ by regularity results for elliptic equations \cite{GT} and 
$W^{2,p}(\Omega) \subset C^1(\overline{\Omega})$ by embedding \cite{Adams75}.
\end{proof}
To conclude the solvability of the original problem \eqref{eq:quasi}--\eqref{eq:quasi_bc}, we use the following statement which follows immediately from the chain rule.
\begin{lemma} \label{lem:equivalence}
Let (A1)--(A3) hold and $A(t) = \int_{\gmin}^t a(w) \d w$ for some $\gmin \in \RR$. 
Then $u \in H^1(\Omega)$ is a solution of \eqref{eq:quasi}--\eqref{eq:quasi_bc}, if, and only if, 
the function $v(x) =  A(u(x))$ solves \eqref{eq:bvp1}--\eqref{eq:bvp2}. 
\end{lemma}

\paragraph{\bf Proof of Theorem~\ref{thm:solvability}:} 
Due to (A3) the function $A$ is continuously differentiable and strictly monotone. 
Using $u(x) = A^{-1}(v(x))$, we get 
\begin{align*}
\nabla u(x) = \frac{1}{A'(A^{-1}(v(x))} \nabla v(x) = \frac{1}{a(u(x))} \nabla v(x),
\end{align*}
which already shows that $u \in W^{1,p}(\Omega) \subset C(\overline{\Omega})$.
By the continuous embedding of $H^1(\RR) \subset C(\RR)$, the parameter $a$ is at least continuous, and therefore $\nabla u$ is continuous as well by the above formula. \hfill \qed

\section{Mapping properties of the forward operator}
\label{sec:map}

\paragraph{\bf Proof of Theorem~\ref{thm:map}:}
Setting $h(s) := g(\gamma(s))$, we get by the transformation formula for integrals
 \begin{align*}
 \|T A\|_{L^2(0,1)}^2
 &= \int_0^1 |A(h(s))|^2 \d s
  = \int_{\gmin}^{\gmax} |A(u)|^2 \frac{1}{|h'(h^{-1}(u))|} \d u \\
 &= \int_{\gmin}^{\gmax} |A(u)|^2  w(u) \d u.
\end{align*}

%
By assumption (A5), we know that $w$ is uniformly positive and bounded from above and below. Consequently,
\begin{align*}
\|T A\|_{L^2(0,1)} = \|A \sqrt{w}\|_{L^2(I)} \approx \|A\|_{L^2(I)}.
\end{align*}
This completes the proof of Theorem~\ref{thm:map}. \hfill \qed \\

The norm equivalence stated in Theorem~\ref{thm:map} shows that the inverse problem 
\eqref{eq:ip} with sufficiently regular exact data is well-posed, 
if $T$ is considered as an operator from $L^2(I)$ to $L^2(0,1)$.
As we will see next, this property can no longer be used in the presence of measurement errors.\\

\paragraph{\bf Proof of Theorem~\ref{thm:perT}:}
Let $T^\delta$ and $y^\delta$ be defined as in \eqref{eq:opT} but with perturbed Dirichlet and Neumann data $g^\delta$ and $j^\delta$. 
As before, define $h(s) = g(\gamma(s))$ and $h^\delta(s) = g^\delta(\gamma(s))$. Then 
\begin{align*}
|(\Td A)(s) - (T A)(s)| = |\int_{h(s)}^{h^\delta(s)} a(u) \d u| \le \|a\|_{L^\infty(I)} |h^\delta(s) - h(s)|. 
\end{align*}
The first estimate now follows from 
$\|a\|_{L^\infty} = \|A\|_{W^{1,\infty}} \le C \|A\|_{H^2}$, where we used the 
definition of $A$ and an embedding theorem. 
 
For the second estimate, let $v$ and $v^\delta$ be the zero average solutions of \eqref{eq:bvp1}--\eqref{eq:bvp2} 
for Neumann data $j$ and $j^\delta$, respectively. Then by the usual a-priori estimates for elliptic problems \cite{GT}, we obtain
\begin{align*}
 \|v - v^\delta\|_{W^{1,3}(\Omega)} \le C' \|j-j^\delta\|_{L^{2}(\partial \Omega)}
\end{align*}
and by standard trace estimates \cite{Adams75} and assumption (A6), we get 
\begin{align*}
\|v - v^\delta\|_{L^2(\gamma)} \le c  \|v - v^\delta\|_{W^{1,3}(\Omega)} \le cC' \delta. 
\end{align*}
This completes the proof of Theorem~\ref{thm:perT}. \hfill \qed

\section{A scale of Hilbert spaces}
\label{sec:hs}

Motivated by Theorem~\ref{thm:perT}, we will consider the operator $T$ as a mapping from 
a subspace of $X \subset H^2(I)$ into $ Y = L^2(0,1)$. This will allow us to control the effect of the measurement errors. 
To completely define the appropriate preimage space $X$, we proceed as follows: 
Consider the operator $L^4 : A \mapsto A''''+A$ with domain
\begin{align*} 
 D(L^4) = \{& A \in H^4(I) : \\
            & A(\gmin) = A''(\gmin) = A''(\gmax) = A'''(\gmax) = 0\}.
\end{align*}
The properties of this operator are summarized in 
\begin{lemma} \label{lem:L4}
The operator $L^4 : D(L^4) \subset L^2(I) \to L^2(I)$ is a densely-defined, self-adjoint, and strictly positive linear operator.
\end{lemma}
\begin{proof}
Let $u,v \in D(L^4)$ and let $(u,v) = \int_I u v \, dx$, then
\begin{align*}
 (L^4 u, v) 
 &= (u'''',v) + (u,v)
  = -(u''',v') + \underbrace{u''' v|_\gmin^\gmax}_{=0} + (u,v) \\
 &= (u'',v'') - \underbrace{u'' v'|_\gmin^\gmax}_{=0} +(u,v) 
  = -(u',v''') + \underbrace{u' v''|_\gmin^\gmax}_{=0} + (u,v) \\
 &= (u,v'''') - \underbrace{u v'''|_\gmin^\gmax}_{=0} + (u,v)
  = (u, L^4 v).
\end{align*}
Thus $L^4$ is formally self-adjoint and, moreover, $D(L^4) = D({L^4}^*)$.
The positivity of $L^4$ results from 
\begin{align*}
 (L^4 u, u) = (u'',u'') + (u,u) > 0,
\end{align*}
which holds for all non-trivial functions $u \in D(L^4)$. 
\end{proof}

By spectral calculus and interpolation \cite{Rudin}, we can now define arbitrary powers of $L^4$. 
In particular, there exists a unique positive $4$th root of $L^4$, which we denote by $L$, with domain
\begin{align} \label{eq:domL}
 D(L) = \{ u \in H^1(I) : u(\gmin) = 0\}.
\end{align}
This is again a densely defined self-adjoint unbounded linear operator. 
In the usual manner \cite{EHN96,KreinPetunin}, we obtain for $s \ge 0$ the Hilbert spaces
\begin{align}
 X_s = D(L^s) = \overline{\bigcap_{n\geq 0} D(L^n)}^{\|\cdot\|_{X_s}},
\end{align}
where the closure is taken w.r.t. to the norm $\|u\|_{X_s} = \|L^s u\|_{L^2(I)}$.
The corresponding spaces for negative index $s < 0$ are defined by duality. 
This yields
\begin{lemma} \label{lem:tXs}
The family $\{X_s\}_{s \in \RR}$ is a Hilbert scale and for  
any $s,t \in \RR$, the operator $L^s : X_{t+s} \to X_t$ is an isomorphism. 
Moreover, for any $r \le s \le t$ with $r < t$ there holds the \emph{interpolation inequality}
\begin{align} \label{eq:intpol}
\|L^s u\|_{X_0} \le \|L^r u\|^\frac{t-s}{t-r}_{X_0} \|L^t u\|^\frac{s-r}{t-r}_{X_0}.
\end{align}
\end{lemma}
The proof of this statements follows more or less directly from the construction; see \cite[Sec.~8.4]{EHN96} for details. 
For our analysis, we will utilize the shifted Hilbert scale consisting of spaces
\begin{align} \label{eq:Xs}
 \tX_s := X_{s+1}.
\end{align}
Note that $\{\tX_s\}_{s \in \RR}$ is again a scale of Hilbert spaces with norms 
\begin{align} \label{eq:normXs}
\|u\|_{\tX_s} := \|L^{s+1} u\|_{L^2(I)}. 
\end{align}
The cases $s \in \{-1,0,1\}$ will be of particular importance for our analysis below. 
Let us therefore summarize some facts about these spaces.
\begin{lemma} \label{lem:X1}
We have $\tX_0 = \{u \in H^1(I) : u(0)=0\}$. Moreover,
\begin{align*}
\tX_{-1} = L^2(I)
\quad \text{and} \quad  
\tX_1 = \{u \in H^2(I) : u(0)=0\}, 
\end{align*}
and the corresponding norms are given by
\begin{align*}
\|u\|_{\tX_{-1}} = \|u\|_{L^2(I)} 
\quad \text{and} \quad
\|u\|_{\tX_1}^2 = \|u''\|^2_{L^2(I)} + \|u\|_{L^2(I)}^2.
\end{align*}
\end{lemma}
\begin{proof}
The assertions about the spaces $\tX_0$ and $\tX_{-1}$ 
follow directly from the construction of the spaces $\tX_s = X_{s+1}$.  
By definition, the norm of the space $\tX_1 = X_2$ is given by
\begin{align*}
\|u\|_{\tX_1}^2 
  &= (L^2 u, L^2 u) = (L^4 u, u) = (u'''',u) + (u,u)\\
  &= (u'',u'') + (u,u) = \|u''\|^2_{L^2(I)} + \|u\|_{L^2(I)}^2,
\end{align*}
which holds for all $u \in X_4 = D(L^4)$. 
The form of the space $\tX_1$ then follows by density of $X_4 \subset X_2=\tX_1$.
\end{proof}

Note that, apart from some boundary conditions, the shifted Hilbert spaces $\tX_s$ coincide with the Sobolev spaces $H^{s+1}(I)$.

\section{Regularization in Hilbert scales}
\label{sec:reghs}

With $T$ and $\Td$ corresponding to exact and perturbed Dirichlet data, and 
with $L$ as in the previous section, let us define operators
\begin{align}
B = T L^{-1} 
\quad  \text{and} \quad 
\Bd = \Td L^{-1}.
\end{align}
The operators are understood as mappings from $X \! = \! \tX_0$ to $Y \! = \!L^2(0,1)$.
The inverse problem \eqref{eq:ip} with exact data is then equivalent to 
\begin{align} \label{eq:iphs}
 B z &= y, \qquad A = L^{-1} z.
\end{align}
Let us next translate the mapping properties of the operator $T$ 
into the Hilbert scale framework.
\begin{lemma} \label{lem:mappingB}
Let (A1)--(A5) hold. Then
\begin{align} \label{eq:mapB}
 c_T \|z\|_{\tX_{-2}} \le \|B z\|_{Y} \le C_T \|z\|_{\tX_{-2}}
\qquad \text{for all } z \in \tX_0.
\end{align}
If, in addition, the bounds (A6)--(A7) on the data noise hold, then
\begin{align} \label{eq:pertB}
 \|(\Bd - B) z\|_Y \le \tilde C \delta \|z\|_{\tX_0}.
\end{align}
\end{lemma}
\begin{proof}
By Theorem~\ref{thm:map}, and with $B=TL^{-1}$ and $z=L A$, we have
\begin{align*}
\|B z\|_Y 
 = \|T A\|_Y \approx \|A\|_{L^2(I)} 
 = \|L^{-1} z \|_{X_{0}} = \|z\|_{\tX_{-2}},
\end{align*}
which already yields the first result. The second assertion follows from
\begin{align*}
\|\Bd z - B z\|_Y 
 &= \|\Td A - T A\|_Y \le C \delta \|A\|_{H^2(I)}  \\
 & \approx \delta \|A\|_{\tX_{1}} = \delta \|z\|_{\tX_{0}}. 
\end{align*}
Here we used that the norm of $H^2(I)$ is equivalent to the norm of $\tX_{1}$ for all functions in $\tX_1$ by Ehrling's lemma and Lemma \ref{lem:X1}.
\end{proof}
%
%
%
%
Tikhonov regularization \eqref{eq:tikhs} for the perturbed problem can now be phrased equivalently as
\begin{align} \label{eq:tik..}
z_\alpha^\delta = \arg\min_{z \in \tX_0} \|\Bd z - y^\delta\|_Y^2 + \alpha \|z\|_{\tX_{0}}^2, \qquad A_\alpha^\delta = L^{-1} z_\alpha^\delta.
\end{align}
Recall that $\tX_0 = X_1$ corresponds to $H^1(I)$ plus boundary condition, so we are essentially regularizing the $H^1(I)$ norm of the function $z$. The error of the regularized solution can then be  bounded by 
\begin{lemma} \label{lem:rateshs}
Let (A1)--(A7) hold. Then for $\alpha \approx \delta^2$ 
\begin{align} \label{eq:rateshs}
 \|z_\alpha^\delta-z\|_{\tX_0} = O(1) 
\quad \text{and} \quad 
 \|B z_\alpha^\delta - B z\|_{Y} = O(\delta).
\end{align}
\end{lemma}
\begin{proof}
Set $\bar y:=\Bd z$ and note that by Lemma~\ref{lem:mappingB}
\begin{align} \label{eq:perty}
 \|y^\delta - \bar y\|_Y \le \|y^\delta - y\|_Y + \|(B-\Bd) z\|_Y \le (1+\tilde C) \delta.
\end{align}
As usual \cite{EHN96}, we define $g_\alpha(\lambda) = \frac{1}{\lambda + \alpha}$ and $r_\alpha(\lambda) = 1-\lambda g_\alpha(\lambda)$.
Then the regularized solution can be expressed as $z_\alpha^\delta = g_\alpha(\Bd^* \Bd) \Bd^* y^\delta$ and we further set $\bar z_\alpha = g_\alpha(\Bd^* \Bd) \Bd^* \bar y$.
Note that $\Bd^*$ denotes the adjoint of $\Bd$ w.r.t. the spaces $\tX_0$ and $Y$ here. 
The regularization error can now be expressed as 
\begin{align}
z_\alpha^\delta - z 
&= (z_{\alpha}^\delta - \bar z_\alpha) + (\bar z_\alpha - z) \label{eq:decomp} \\ 
&= g_\alpha (\Bd^*\Bd) \Bd^* (y^\delta - \bar y) - r_\alpha(\Bd^* \Bd) z. \notag
\end{align}
The standard spectral estimates yield $\|g_\alpha (\Bd^*\Bd) \Bd^*\| \le C \alpha^{-1/2}$ and $\|r_\alpha(\Bd^* \Bd)\| \le C$, where $C$ only depends on the norm of $\Bd$.
Together with \eqref{eq:perty}, this yields
\begin{align} \label{eq:est0}
\|z_\alpha^\delta-z\|_{\tX_0} \le C \alpha^{-1/2} (\tilde C+1)  \delta + C \|z\|_{\tX_0} = O(1), 
\end{align}
where we used the parameter choice $\alpha \approx \delta^2$ for the last step.

Let us now estimate the residual $\|B z_\alpha^\delta - B z\|_{Y}$:
Using the error decomposition \eqref{eq:decomp}, the estimate \eqref{eq:pertB} for the perturbation, and spectral estimates, we get
\begin{align*}
&\|B z_{\alpha}^\delta - B z \|_Y  
\le\|(B-\Bd) (z_\alpha^\delta - z)\|_Y + \|\Bd z_\alpha^\delta - \Bd z\|_Y \\
&\quad 
  \preceq \delta \|z_\alpha^\delta - z\|_{\tX_0}
  + \|\Bd g_\alpha (\Bd^*\Bd) \Bd^* (y^\delta - \bar y)\|_Y 
 + \|\Bd r_\alpha(\Bd^* \Bd) z\|_Y \\
&\quad 
  \preceq \delta \|z_\alpha^\delta - z\|_{\tX_0} + \|y^\delta - \bar y\|_Y + \alpha^{1/2} \|z\|_{\tX_0}. 
\end{align*}
For ease of notation, we wrote $a \preceq b$ with the meaning $a \prec C b$ for some constant $C$ here. The result now follows by using the bound \eqref{eq:est0}, the estimate \eqref{eq:perty}, and the a-priori choice $\alpha \approx \delta^2$ for the regularization parameter.
%
%
\end{proof}
The spectral equivalence \eqref{eq:mapB} and an interpolation argument yield
\begin{lemma} \label{lem:esths}
Under the assumptions of Lemma~\ref{lem:rateshs}, there holds
\begin{align} \label{eq:esths}
\|z_\alpha^\delta - z\|_{\tX_{-2r}} = O(\delta^{r}) 
\qquad \text{for all} \quad 0 \le r \le 1. 
\end{align}
\end{lemma}
\begin{proof}
Note that 
$\|B (z_\alpha^\delta-z)\|_Y \approx \|z_\alpha^\delta-z\|_{\tX_{-2}}$. The result now follows from the previous lemma and the interpolation inequality \eqref{eq:intpol}.
\end{proof}

\paragraph{{\bf Proof of Theorem~\ref{thm:rates} and \ref{thm:main}}}
Using that $A_\alpha^\delta = L^{-1} z_\alpha^\delta$, Theorem~\ref{thm:rates} is a direct consequence of Lemma~\ref{lem:esths} with $r=0$ and $r=1$, respectively. Theorem~\ref{thm:main} finally follows by differentiation and interpolation. \hfill \qed

\section{Discussion and numerical illustration}
\label{sec:num}

Before turning to numerical tests, let us shortly discuss some generalizations
 that will explain which results we can expect.

\subsection{Remarks and generalizations} \label{sec:gen}

The convergence results of Theorem~\ref{thm:rates} can be strengthened as follows: If $z \in \tX_u$ for some $0 \le u \le 1$, then for the parameter choice $\alpha \approx \delta^{\frac{4}{u+2}}$, one has
\begin{align*} 
 \|z_\alpha^\delta - z\|_{\tX_u} = O(1)     
\quad \text{and} \quad  
 \|z_\alpha^\delta - z\|_{\tX_0} = O(\delta).
\end{align*}
For $\Bd=B$, corresponding statements can be found in \cite{Egg06,Nat84}. 
In the presence of perturbations of the operator, the proof of these estimates follows similar to that of Lemma~\ref{lem:rateshs}. Here, we additionally employ the spectral equivalence $L^u \sim (B^* B)^{u/4}$, which follows from Lemma~\ref{lem:mappingB}, and the estimate 
$$
\|(B^* B)^\mu - (\Bd^* \Bd)^\mu\| \preceq \|B-B^\delta\|^\mu \le c_\mu \delta^\mu,
$$ 
which holds for $|\mu| \le 1/2$ and can be deduced from \eqref{eq:pertB} by interpolation; see \cite{Egg06,KalNeuSch08} for details. 
Using the relation between $a$, $A$, and $z$, and the definition of the norms of $X_s$ and $\tX_s$, we then obtain by interpolation
\begin{align*}
\|a_\alpha^\delta-a\|_{L^2(I)} \approx   \|A_\alpha^\delta-A\|_{X_1}  = \|z_\alpha^\delta - z\|_{\tX_{-1}} = O(\delta^{\frac{u+1}{u+2}}).
\end{align*}
For $u=0$ we recover the result of Lemma~\ref{lem:rateshs} and Theorem~\ref{thm:main}.

As the following argument shows, the optimal rates that one can expect in practice will however be limited:
A condition $z \in \tX_u$ amounts to $A \in X_{u+2} \subset H^{u+2}(I)$. 
If we require this for $u>1/2$, we will implicitly impose additional 
boundary conditions $A''(\gmin) = A''(\gmax)=0$, which leads to 
the assumption $a'(\gmin) = a'(\gmax)=0$ on the parameter $a$. Similar conditions have been used in \cite{Kuegler03} to derive the convergence statement. Note, however, that such conditions will usually not be valid in practice and we can therefore not expect a source condition $z \in X_u$ to hold 
for $u > 1/2$. 
The best possible rate that can be expected for smooth parameters, will then be 
$
\|a_\alpha^\delta - a\|_{L^2} 
= O(\delta^{3/5}), 
$
which results for $u = 1/2$ and a parameter choice $\alpha = \delta^{8/5}$. 
We will in fact observe these rates in our numerical tests.  

Let us finally mention, that the regularization parameter $\alpha$ can also be chosen a-posteriori, e.g., by the discrepancy principle. This follows from the estimates for the residual in Lemma~\ref{lem:rateshs}; see \cite{Neu88} for details.

\subsection{Setup and discretization of the test problem}

Let us now describe the setup of our numerical tests. 
We let $\Omega$ be the unit ball and consider the curve
\begin{align*}
\gamma : \big[\frac{\pi}{4},\frac{3\pi}{4} \big] \to \RR^3,  
\qquad s \mapsto (\cos s, \sin s,0)
\end{align*}
for taking the measurements.
We assume that $u$ is given as the solution of \eqref{eq:quasi} with Dirichlet 
boundary conditions $u(x,y,z) = x$ on $\partial \Omega$. 
The interval of states that are attained on $\gamma$ then is $I=[\gmin,\gmax] = [-\frac{1}{\sqrt{2}},\frac{1}{\sqrt{2}}]$.  
As exact parameter we choose $a(u)=1+u^2$. Its antiderivative is given by $A(u)=u+u^3/3+C$ with integration constant $C$ such that $A(-\frac{1}{\sqrt{2}})=0$.
The exact data for \eqref{eq:bvp3} are then given by
\begin{align*}
 h(s) = g(\gamma(s)) = \cos(s)
 \quad \text{and} \quad
 y(s) = A(h(s)).
\end{align*}
Note that according to \eqref{eq:bvp3}, these are the only data that are required for the solution of the inverse problem. 

For the numerical solution, we employ a discretized version of the Tikhonov functional \eqref{eq:tikhs}.
The integral required for the data misfit term  is approximated by numerical quadrature with $500$ integration points and the required data $h^\delta$ and $y^\delta$ are computed by evaluation of the exact functions $h$ and $y$ on the corresponding integration points and by adding uniformly distributed random noise of size $\delta$.
Let us note that besides the perturbations that were added deliberately to the data, our numerical tests also suffer from discretization errors which may become dominant when we proceed to small noise levels. 
The function $a$ is represented by a continuous piecewise linear spline $a_h$ over a uniform partition of the interval $I^\delta = [\min h^\delta, \max h^\delta]$
into $200$ elements. The antiderivative $A_h$ is computed by numerical integration on a finer discretization.

\subsection{Numerical test results}  

To illustrate the validity of Theorem~\ref{thm:main}, we will report on the errors 
$$
{\rm err}_0 = \|a_\alpha^\delta - a\|_{L^2(I)} \quad \text{and}\quad {\rm err}_1 = \|a_\alpha^\delta-a\|_{H^1(I)}
$$
obtained in our numerical tests.
Note that by definition of $A=\int a \d u$ and $z=L A$, and of the norms of $\tX_s$, we have 
${\rm err}_0 \approx \|z_\alpha^\delta - z\|_{\tX_{-1}}$ and ${\rm err}_1 \approx \|z_\alpha^\delta - z\|_{\tX_0}$.
For the a-priori parameter choice $\alpha=\delta^2$, we therefore 
expect convergence rates $\text{err}_1 = O(1)$ and $\text{err}_0=O(\delta^{1/2})$. 
In view of the remarks of Section~\ref{sec:gen} and the smoothness of the exact parameter $a$, we can in fact obtain a better rate $\text{err}_0=O(\delta^{3/5})$ and $\text{err}_1=O(\delta^{1/5})$ with a parameter choice $\alpha \approx \delta^{8/5}$.
The errors observed in our numerical experiments are listed in Table~\ref{tab:res}.
For completeness, we also list the residual errors $\text{res}=\|T^\delta A_\alpha^\delta - y^\delta\|$ which should be in the order of the noise level $\delta$. 

\begin{table}[ht]
\begin{center}
 {\small 
 \begin{tabular}{c||c|c|c||c|c|c} 
$\delta$ & err$_0$ & err$_1$ & res & err$_0$ & err$_1$ & res\\ 
 \hline
 1e-2    &  0.036672  & 0.51336  &  0.010445 & 0.037113 & 0.50677 & 0.101159\\
 1e-3    &  0.008765  & 0.32007  &  0.009861 & 0.009421 & 0.32653 & 0.001003\\
 1e-4    &  0.002147  & 0.22529  &  0.000094 & 0.002691 & 0.21059 & 0.000098\\
 1e-5    &  0.001071  & 0.23345  &  0.000008 & 0.000792 & 0.13419 & 0.000009\\
 1e-6    &  0.000235  & 0.06543  &  0.000001 & 0.000245 & 0.05828 & 0.000001\\
 \end{tabular}
}
 \medskip
 \setlength\captionindent{1.4em}
 \caption{\label{tab:res} \small Errors and residual observed in numerical tests for the parameter choice $\alpha=\delta^2$ (left) and $\alpha = 0.1 \times \delta^{8/5}$ (right). 
 The results for noise level $\delta=10^{-6}$ are affected by inverse crime and change when  
 repeating the test with a finer discretization.}
\end{center}
\end{table}
As expected, we cannot see a clear convergence behavior 
for the error in the $H^1$ norm for the parameter choice $\alpha \approx \delta^2$, while we observe a clear convergence of the error in the $L^2$ norm.
The small errors of the last line can be explained by an inverse crime, i.e., the finite dimension of the discretization limits the error propagation. This effect vanishes when refining the grid used for the discretization. 
The parameter choice $\alpha \approx \delta^{8/5}$ yields
a better convergence behavior for both, the error in the $H^1$ and the $L^2$ norm. This is fully explained by our remarks made in Section~\ref{sec:gen}. 
In both cases, the residual is in the order of the noise level. 
For a better comparison, we also display the numerical results 
and the theoretical rates in Figure~\ref{fig:res} in a double logarithmic plot.
\begin{figure}[ht]
\begin{center}
\includegraphics[width=0.48\textwidth]{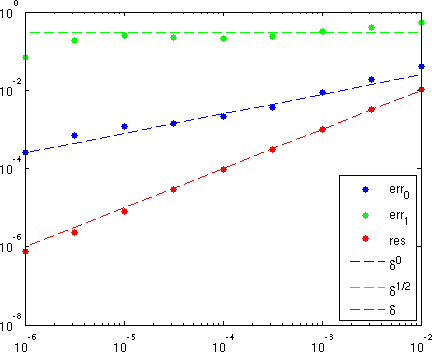} 
\hfill
\includegraphics[width=0.48\textwidth]{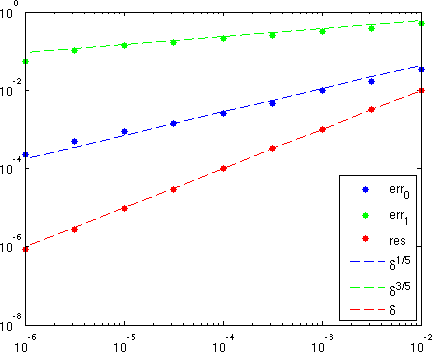} 
\setlength\captionindent{1.4em}
\caption{\label{fig:res} \small Convergence plots for the error in $H^1$ and $L^2$ norm and the residual for parameter choices $\alpha=\delta^2$ (left) and $\alpha=\delta^{8/5}$ (right). The dashed lines illustrate the theoretical rates.}
\end{center}
\end{figure}
One can clearly see that the numerical results are in good agreement with the 
theoretical rates. As mentioned above, the errors for the smallest noise level 
are biased by the finite dimension of the numerical experiment.

\section{Summary}

We considered the determination of the diffusion parameter in a quasi-linear elliptic equation from 
a single set of overspecified boundary data. Based on a reformulation as a linear inverse problem and an explicit characterization of the mapping properties of the forward operator, we were able to employ Hilbert scale regularization. This allowed us to obtain convergence rates in weaker norms without additional source conditions.

\section*{Acknowledgments}
HE acknowledges support by DFG via Grant IRTG 1529 and GSC 233.
The work of JFP was supported by DFG via Grant 1073/1-1 and from the Daimler and Benz Stiftung via Post-Doc Stipend 32-09/12. 

\bibliographystyle{plain}
\bibliography{bib}

\end{document}